\documentclass[final]{siamltex}

\usepackage{amsmath,amsfonts,amssymb}
\usepackage{mathrsfs,seceqn}
\usepackage[sort&compress,numbers]{natbib}
\usepackage{graphicx,float,color}

\newtheorem{remark}[theorem]{Remark}
\newcommand{\Hm}{X_\mathbf{m}}
\newcommand{\intd}{\int_\Omega}
\newcommand{\wto}{\rightharpoonup}

\newcommand{\R}{\mathbb{R}}
\newcommand{\p}{\partial}

\newcommand{\ds}{\displaystyle}
\newcommand{\pot}{{\phi_\eps}}
\newcommand{\potp}{{\phi'_\eps}}
\newcommand{\Pot}{\Phi}
\newcommand{\PotG}{W}
\newcommand{\potG}{{W_\eps}}
\newcommand{\phim}{|\Phi(1)|}

\newcommand{\Real}{\mathbb{R}}

\newcommand{\intdx}{\ {\rm d}x}
\newcommand{\dom}{{\Omega}}
\newcommand{\Ldom}{}
\newcommand{\zdiff}{h}
\newcommand{\eps}{\varepsilon}

\newcommand{\lb}{\label}

\newcommand{\go}{\rightarrow}

\newcommand{\la}{\lambda}

\newcommand{\pt}{\partial_t}

\newcommand{\px}{\partial_x}

\newcommand{\pxx}{\partial_{xx}}

\newcommand{\pxxx}{\partial_{xxx}}
\newcommand{\ee}{\end{equation}}
\newcommand{\be}{\begin{equation}}
\newcommand{\bea}{\begin{eqnarray}}
\newcommand{\eea}{\end{eqnarray}}
\newcommand{\sbea}{\begin{subequations}\begin{eqnarray}}
\newcommand{\seea}{\end{eqnarray}\end{subequations}} 
\newcommand{\ees}{\end{equation*}}
\newcommand{\bes}{\begin{equation*}}
\newcommand{\beas}{\begin{eqnarray*}}
\newcommand{\eeas}{\end{eqnarray*}}
\newcommand{\nn}{\nonumber}
\newcommand{\rf}[1]{(\ref{#1})}

\newcommand{\mR}{\mathbb{R}\,}

\begin{document}

\title{Stationary solutions of liquid two-layer thin film models\thanks{Part of this work was supported
by DFG priority program SPP 1506 \textit{Transport Processes at Fluidic Interfaces}}}

\author{S.~Jachalski\footnotemark[2] \and
Robert Huth\footnotemark[2] \and
Georgy Kitavtsev\footnotemark[3] \and
Dirk Peschka\footnotemark[2] \and 
Barbara Wagner\footnotemark[4]
}

\maketitle


\pagestyle{myheadings}
\thispagestyle{plain}
\markboth{Jachalski et al.}{Stationary solutions of liquid two-layer films}

\renewcommand{\thefootnote}{\fnsymbol{footnote}}
\footnotetext[2]{Weierstrass Institute, Mohrenstra{\ss}e 39, 10117 Berlin, Germany}
\footnotetext[3]{Max Planck Institute for Mathematics in the Sciences, Inselstra{\ss}e 22, 04103 Leipzig, Germany}
\footnotetext[4]{Technische Universit\"at Berlin, Institute of Mathematics, Stra{\ss}e des 17. Juni 136, 10623 Berlin, Germany}

\begin{abstract}
We investigate stationary solutions of a thin-film model for liquid two-layer  flows in an energetic formulation that is motivated by its gradient flow structure. The goal is to achieve a rigorous understanding of the contact-angle conditions for such two-layer systems. We pursue this by investigating a  corresponding energy that favors the upper liquid to dewet from the lower liquid substrate, leaving behind a layer of thickness $h_*$. After proving existence of stationary solutions for the resulting system of thin-film equations we focus on the limit $h_*\to 0$ via matched asymptotic analysis. This yields a  corresponding sharp-interface model and a matched asymptotic solution that includes logarithmic switch-back terms. We compare this with results obtained using $\Gamma$-convergence, where we establish existence and uniqueness of energetic minimizers in that limit.
\end{abstract}
\section{Introduction} 

Understanding stability and dewetting behaviour of thin liquid films coating a solid or a liquid substrate is important in many technological applications and natural phenomena on the micro- to nano scale. They range from tear films of the human eye to organic photovoltaics to numerous applications in the polymer based semiconductor industry.

Typical film thicknesses for these applications may range from tens to hundreds of nanometers and, depending on the material composition, may be susceptable to rupture and formation of holes due to intermolecular forces. Such rupture processes typically initiate complex dewetting scenarios, where holes grow further and their trailing rims merge into polygonal networks which eventually decay into patterns of droplets, that evolve on a slow time scale towards a global minimal energy state. 

The present study focusses on liquid substrates, that energetically favor an interface with the underlying solid. In this case the stages of the dewetting process for the upper liquid proceed to some extend in parallel to those exhibited during dewetting of a liquid film from a solid substrate. The latter system has been investigated much more intensely in recent decades, both, experimentally and theoretically. Examples of the complex pattern formation can be found in Sharma et al. \cite{SR96} or Seemann et al. \cite{SHJ01b} and further experimental and theoretical investigations in numerous references in the recent reviews by Craster \& Matar \cite{CO09} and Herminghaus et al. \cite{HBS08}. 

For liquid-liquid dewetting, experimental studies depicting some of these dewetting stages have been conducted by various groups such as by Segalman \& Green \cite{segalman1999dynamics}, Lambooy et al. \cite{lambooy1996dewetting}, Slep et al. \cite{slep00} or Wang et al. \cite{wang2001dewetting} for the standard system of liquid polystyrene (PS) on a liquid polymethylmethacrylate(PMMA) substrate. They include investigations of rupture and hole growth, dewetting dynamics and equilibrium contact angles the liquid droplets make with the underlying liquid layer, where now the contact line is fixed by two angle conditions instead of one, i.e. Youngs  law is replaced by the Neumann triangle construction \cite{neumann1894vorlesung}. Following the pioneering study by  Brochard-Wyart et al.~\cite{brochard1993liquid}, where various dewetting regimes are derived and analysed, stability of liquid-liquid systems were investigated by Danov et al.~\cite{danov1998stability}, Pototsky et al.~\cite{Pototsky2004Alternative}, Golovin \& Fisher \cite{fisher2005nonlinear}. Stationary states and the dynamics towards stationary states were studied by Pototsky et al. \cite{pototsky2005morphology}, by Craster \& Matar \cite{craster2006dynamics} and by Bandyopadhyay \& Sharma \cite{bandyopadhyay2006nonlinear} and for the case of gravity-driven liquid droplets on a inclined liquid substrate by  Kriegsmann \& Miksis \cite{kriegsmann2003steady}. 

Interestingly, direct quantitative comparisons of theoretical with experimental results, in particular on micro- and nano-scale, regarding for example the morphology of the interfaces, as performed in Kostourou et al. \cite{kostourou2010interface}, or equilibrium values of the Neumann triangle, as discussed in \cite{slep00}, still leave many issues in need to be explained, such as  the dependency of the morphology of the interfaces on the rheology of the liquids, their layer thicknesses or material parameters. On the other side, even for the simplest mathematical models of Newtonian two-layer liquid systems, mathematical theory is still largely open and this is the main focus of the present study. 

Here, we are guided by the many similarities to dewetting from a solid substrate, and expect that some of the mathematical analysis developed for liquid films dewetting from a solid substrate can be carried over to liquid substrates. In particular, we aim to extend the existing theory for liquid droplets on a solid substrate to the situation on a liquid substrate.

As a starting point we recall the work by Bertozzi et al. \cite{bertozzi2001dewetting} regarding stationary states and coarsening of droplets on solid substrates, where they showed existence of smooth global solutions for positive data with bounded energy for the no-slip lubrication equation. In addition they prove existence of global minimizers and determine a family of positive periodic solutions for admissible intermolecular potentials consisting of long-range attractive and short range repulsive contribution, and investigate their linear stability. 
Further extensions were given in Laugesen \& Pugh \cite{laugesen2000linear}, where linear stability of stationary solutions for the thin film equation with Neumann boundary conditions or periodic boundary conditions was investigated. 
Extensions of the existence theory to thin film equations accounting for slip at the liquid-solid interface were given in Rump et al.~\cite{otto2007coarsening} and Kitavtsev et al.~\cite{KRWP11}. Convergence to stationary
solutions of the one-dimensional thin-film equation and the number of
stationary states was recently investigated by Zhang \cite{zhang2009counting}. 

The extension of the existence theory to two-layer liquid systems is given in Section 3, after the formulation of the problem. With the appropriate energy functional for the two-layer system of coupled thin-film equations for the interfacial heights $h_1$ and $h_2$ (see sketch in figure \ref{fig:sketch}) together with Neumann boundary conditions, we show existence of smooth stationary solutions as well as existence of a global minimizer for the steady state problem. 

Admissible intermolecular potentials for the liquid-liquid system are of the form

\begin{align}
\phi(h_2-h_1)=\frac{\phi_*}{\ell-n}\left[\ell\left(\frac{h_*}{h_2-h_1}\right)^{n}
-n\left(\frac{h_*}{h_2-h_1}\right)^{\ell}\right],
\end{align}
where $h_1$ is the height of the liquid-liquid interface, $h_2$ the height of the free surface and its minimal value $\phi_*<0$ is attained at $h_*$. We note that such potentials are widely used in the literature, see e.g. the review \cite{ODB97}.  A choice that is related to the standard Lennard-Jones potential and typically used in experiments is $(n,\ell)=(2,8)$, see e.g. \cite{SHJ01b}.

\begin{figure}[H]
\centering
\includegraphics*[width=0.7\textwidth]{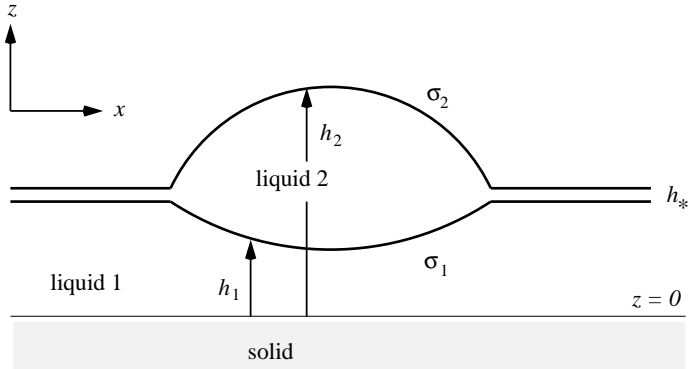}
\caption{ Sketch of liquid droplet with surface tension $\sigma_2$ between air and liquid\,$_2$ on top of a \hspace*{2cm} liquid layer with interfacial  tension $\sigma_1$ between liquid\,$_2$ and liquid\,$_1$.}
\label{fig:sketch}
\end{figure}

The Neumann triangle construction for contact angles at a triple junction is thereby replaced by properties of approximate contact angles  resulting from the particular structure of the surface free energy $\phi$. Starting from this energetic formulation we seek to understand in Section 4 and 5 how the Neumann triangle construction is attained as a limit $h^*\to 0$. This in mind, we first derive in Section 4 a sharp-interface model in the limit of  $h^*\to 0$ using matched asymptotic analysis. This yields the appropriate expression for the equilibrium contact angles of these droplet solutions and as a result the corresponding Neumann triangle. 

We find that, while the equilibrium contact angle is easy to obtain, as expected, the complete asymptotic argument needs to include logarithmic switch-back terms. We note, in retrospect, since equilibrium droplet solutions for solid substrates can be considered as limiting cases for liquid lenses, similar terms should also appear in the matched asymptotic derivation for that problem. 

Finally, in Section 5 we rigorously show existence and uniqueness of the limit $h_*\to 0$ within the framework of $\Gamma$-convergence. We obtain a sharp-interface problem for which we compute the Euler-Lagrange equations, from which one can immediately read off the contact angles. Existence and uniqueness of minimizers is shown here using a rearrangement inequality.


\section{Formulation}
We consider a layered system of two immiscible Newtonian liquids with negative spreading coefficient $\phi_*$. We assume the layered system lives in the two phases $\Omega_1$ and $\Omega_2$ defined by 
\begin{subequations}
\begin{align}
\Omega_1(t)&=\{(x,y,z)\in\Real^3:0 \le z < {h_1}(t,x,y)\},\\
\Omega_2(t)&=\{(x,y,z)\in\Real^3:{h_1}(t,x,y) \le z < {h_2}(t,x,y)\},
\end{align}
\end{subequations}
for all $t>0$ as sketched (in 2D) in figure~\ref{fig:sketch}. Typical applications use liquids such as PMMA for the liquid substrate $\Omega_1$ and PS as the upper liquid $\Omega_2$, both on a scale where gravity can be neglected and, for this study, unentangled and density matched. These simplifying assumptions allow us that the flow of the viscous and incompressible liquids in each phase $\Omega_i$ ($i=1,2$) is governed by the Stokes equation and the continuity equation: 
\begin{subequations}
\begin{align}
-\nabla\cdot\bigl(-p_i\mathbb{I}+\mu_i (\nabla\mathbf{u}_i+\nabla\mathbf{u}_i^\top)\bigr)&=\mathbf{f}_i,\\
\nabla\cdot\mathbf{u}_i&=0, 
\end{align}
\end{subequations}
together with a kinematic condition at each free boundaries $z=h_i$, i.e. 
\be
(\mathbf{e}_z\partial_t h_i - \mathbf{u}_i)\cdot\mathbf{n}_i=0. 
\ee 
Here, $\mathbf{n}_i$ denotes the outer normal.

At the solid substrate a no-slip and an impermeability condition are imposed. At the liquid-liquid interface the velocity is continuous, i.e. $\mathbf{u}_2=\mathbf{u}_1$, and the tangential stress is continuous and the normal stress makes a jump proportional to the mean curvature $[[-p_i\mathbb{I}+\mu_i (\nabla\mathbf{u}_i+\nabla\mathbf{u}_i^\top)]]_{1,2}\mathbf{n}_1=2\sigma_{1,2}\kappa\mathbf{n}_1$. The dewetting process is driven by the intermolecular potential of the upper liquid layer, i.e. $\mathbf{f}_2=-\nabla\phi$. We assume that the thickness of the lower layer is sufficiently thick, so other contributions to the intermolecular potential are negligible, i.e. $\mathbf{f}_1=0$.

In addition we assume that the the ratio $\eps_{\ell}=H/L$ of the vertical to horizontal length scales is always small and the two-layer system can be approximated by a thin-film model. We denote by $L$ the length scale of the typical horizontal width and $H=h_{\rm max}$ the maximum of the difference $h_2-h_1$. 

Detailed derivations of thin-film models for liquid-liquid systems are given for example in \cite{kriegsmann2003steady}, \cite{Pototsky2004Alternative} or recently, by accountung for interfacial slip, in \cite{JMPW11}. For the convenience of the reader we note here the choice of non-dimensional variable and parameters, where the non-dimensional horizontal and vertical coordinates are given by $\tilde x=x/L$, $\tilde y=y/L$ and $\tilde z=z/h_{\rm max}$, respectively, and $\tilde h=h/h_{\rm max}$. The non-dimensional pressures $\tilde p_i=p_i/P$ and the derivative of the non-dimensional intermolecular potential $\phi'_{\eps}=\phi'/P$ are scaled such that $P=\phi_*n\ell/\bigl((\ell-n)h_{\rm max}\bigr)$ and hence 
\begin{align}
\label{phi}
\pot'(\tilde{h})=\frac{1}{\eps}\left[\left(\frac{\eps}{\tilde{h}}\right)^{n+1}
-\left(\frac{\eps}{\tilde{h}}\right)^{\ell+1}\right]
\end{align}
attains the minimal value $\min\pot=-1$ at $\tilde h=\eps$, where
$\eps=h_*/h_{\rm max}$ is the non-dimensional thickness of the ultra-thin film.
For the dynamic problem, the velocities are scaled with the characteristic
horizontal velocity of the dewetting upper layer, such that for the
non-dimensional horizontal and vertical velocities we have $\tilde u= u/U$,
$\tilde v=v/U$ and $\tilde w=w/\eps_{\ell}U$, respectively, with
$U=\eps_{\ell}^3\sigma_2/\mu_2 $, and the non-dimensional time $\tilde t=(U/L)
t$.

For the remainder of this paper it is convenient to introduce the ratios
$\sigma=\sigma_1/\sigma_2$ and $\mu=\mu_1/\mu_2$ of surface tensions and
viscosities, respectively, and drop all the ``$\sim $''.
Within this approximation the normal and tangential stress conditions at the
free surface $h_2$ and at the free liquid-liquid interface $h_1$ yield the
expressions for the pressures $p_2$ and $p_1$ 
\begin{align}\label{pressures}
p_1=-\sigma\Delta h_1 -\ds\potp(h_2-h_1),\quad\qquad p_2=-\Delta h_2 +\potp(h_2-h_1),    
\end{align}
respectively.
Under these assumptions the coupled system of nonlinear fourth order partial differential equations for the profiles of the free surfaces $h_1$ and $h_2$ takes the form  
\begin{align}
\label{sys}
\pt {\bf h}&=\nabla\cdot\left({ Q}\cdot \nabla{\bf p}\right),
\end{align}
where ${\bf h}=(h_1,h_2)^\top$ is the vector of liquid-liquid interface profile 
and liquid-air surface profile. The components of the vector ${\bf
p}=(p_1,p_2)^\top$ are the interfacial pressures given in \rf{pressures}. 
The gradient of the pressure vector is multiplied by the mobility
matrix $Q$ which is given by
\begin{align}
\label{mobmatrix}
{Q}=\frac{1}{\mu}\left[\begin{array}{cccc}
\ds\frac{h_1^3}{3}& \;\ds\frac{h_1^3}{3} + \frac{h_1^2(h_2-h_1)}{2}\\ \\
\ds\frac{h_1^3}{3} + \frac{h_1^2(h_2-h_1)}{2}& \quad\; \ds\frac{\mu}{3}(h_2-h_1)^3 + h_1h_2(h_2-h_1) +
\frac{h_1^3}{3}
\end{array}\right].
\end{align}

We proceed to first show existence of stationary solution to the system 
\rf{phi}-\rf{mobmatrix} with Neumann and Dirichlet boundary conditions on a finite domain $\Omega=(0,L)\subset\R$.

\section{Energy functionals and existence of stationary solutions}

Consider the two-layer thin film equations \rf{pressures}-\rf{mobmatrix} defined on $\Omega=(0,L)\subset\R$ 
with\\ Neumann boundary conditions:
\be
\px h_1=\px h_2=\pxxx h_1=\pxxx h_2=0\ \ \text{for}\ \ x\in\partial\Omega.
\lb{NBC}
\ee
The energy functional associated to the gradient flow of the lubrication equation is given by
\be
E_\eps(h_1,h_2)=\int_0^L\left[\frac{\sigma}{2}\left|\px
h_1\right|^2+\frac{1}{2}\left|\px h_2\right|^2+\pot(h_2-h_1)\right]dx
\lb{DE}
\ee
where the potential function $\pot$ is given as in \rf{phi} with $(n,\ell)=(2,8)$. The relation to the thin-film equations is 
$p_i=\delta E_\eps/\delta h_i$.
From \rf{pressures}-\rf{mobmatrix} and \rf{NBC} conservation of mass follows
\begin{subequations}
\begin{align}
\lb{M1}
\int_{\Omega}h_1(t,x)\,dx&= m_1,\\
\lb{M2}
\int_{\Omega}\left(h_2(t,x)-h_1(t,x)\right)\,dx&= m_2\ \ \text{for all}\
\ t>0,
\end{align}
\end{subequations}
where $m_1$ and $m_2$ are positive constants. 
Any stationary solution of \rf{pressures}-\rf{mobmatrix} with \rf{NBC} satisfies 
\be\label{steadyeqh1h2}
0=Q\cdot\px{\bf p}
\ee
in $\Omega$, where the mobility matrix $Q$ is not singular
i.e. ${\rm det} Q\neq 0$ for all $h_1,\,{h_2-h_1}>0$. Therefore, one obtains that any
positive stationary solution of \rf{pressures}-\rf{mobmatrix} satisfies $\px p_1 =\px p_2 = 0$ in $\Omega$. 
This in turn is equivalent to 
\begin{subequations}
\begin{align}
\sigma\pxx h_1&=-\phi_\epsilon'(h_2-h_1)-\la_2+\la_1,\\
\pxx h_2&=\phi_\epsilon'(h_2-h_1)-\la_1,
\end{align}
\lb{SS1}
\end{subequations}
where constants $\la_2$ and $\la_1$ are Lagrange multipliers associated
with conservation of mass \rf{M1} and \rf{M2}, respectively.
To solve \rf{SS1} let us consider the equation for the difference \[ h(t,x)=h_2(t,x)-h_1(t,x) \]
which reads as follows
\be
\pxx h=\frac{\sigma+1}{\sigma}\phi_\epsilon'(h)+\frac{1}{\sigma}\la_2-\frac{\sigma+1}{\sigma}\la_1.
\lb{DE1}
\ee
For brevity we set
\bes
P:= -\frac{1}{\sigma}\la_2+\frac{\sigma+1}{\sigma}\la_1.
\ees
According to \cite{bertozzi2001dewetting} for  positive $P$, 
there exists a so called {\it droplet} solution $\bar{h}$ to \rf{DE1}  satisfying
boundary conditions \rf{NBC}, such that $\bar{h}(y+L/2)$ is an even function
and monotone decreasing for $y\in (0,L/2)$.

For this solution the asymptotics and main properties are derived in the next section, here we consider $\bar{h}$ as a known analytical function and integrate 
\rf{SS1} two times w.r.t. $x$. We then obtain  a solution to  \rf{SS1} with \rf{NBC} in the form
\bea
h_1=-\frac{1}{\sigma+1}\bar{h}-\frac{1}{2}\frac{\la_2}{\sigma+1}x^2+Cx+C_1,\nonumber\\
h_2=\frac{\sigma}{\sigma+1}\bar{h}-\frac{1}{2}\frac{\la_2}{\sigma+1}x^2+Cx+C_1.
\lb{RSS1}
\eea
Using now again \rf{NBC} one obtains that $\la_2=0,\ C=0$ and
\bea
h_1=-\frac{1}{\sigma+1}\bar{h}+C_1,\nonumber\\
h_2=\frac{\sigma}{\sigma+1}\bar{h}+C_1.
\lb{RSS2}
\eea
The additive constant $C_1$ and the remaining Lagrange multiplier 
are determined from the conservation of masses \rf{M1} and \rf{M2},
respectively. We conclude that the solution \rf{RSS2} is given
by combination of two symmetric droplets with constant outer layer. The next
theorem establishes existence of a global minimizer to the energy functional
\rf{DE} and shows that it satisfies  \rf{SS1} with \rf{NBC}.
\\

\begin{theorem}
Let $\Omega$ be a bounded domain of class $C^{0,1}$ in $\mR^d,\,d\ge 1$ and let $\mathbf{m}=(m_1,m_2)$ with $m_1,m_2>0$. 
Then a global minimizer of $E_{\epsilon}(\cdot,\cdot)$ defined in \rf{DE} exists in the class
\begin{equation}
\label{eqn:spacedef}
\Hm:=\left\{(h_1,h_2)\in H^1(\Omega)^2:
	m_1=\intd h_1,\
	m_2=\intd (h_2-h_1),\
	h_2\ge h_1 
\right\},
\end{equation}

For $d=1$ and $\Omega= (0,\,L)$ the function ${h_2-h_1}$ is strictly positive
and $(h_1,h_2)$ are smooth solutions to the ODE system  \rf{SS1} with \rf{NBC} and
\begin{equation}
\lambda_1=\frac{1}{L}\int_\Omega\phi_\epsilon'(h_2-h_1)\,dx,\quad \lambda_2=0.
\lb{l} 
\end{equation} 
\lb{TMin}
\end{theorem}
\begin{proof} Even though the proof proceeds very analogously to the one of Theorem 2-3
in~\cite{bertozzi2001dewetting} using direct methods of the calculus of variations,  
for the convenience of the reader we give here for our system. 

Let
$(h_1^k,h_2^k)_{k\in\mathrm{N}}$ be a minimizing sequence which exists since
$E_\epsilon(m_1,m_2+m_1)<\infty$. Observe that $\phi_\epsilon(\cdot)$ is
bounded from below by a constant. Hence, a constant $C_2$ exists such that
\begin{equation}
\int_\Omega|\px h_1^k|^2+|\px h_2^k|^2\,dx\le C_2\ \text{for all}\ k\in\mathrm{N}.
\lb{M1ref}
\end{equation} 
Rellich's compactness theorem implies that a subsequence (again denoted by $(h_1^k,h_2^k)_{k\in\mathrm{N}})$
exists which converges strongly in $L^2(\Omega)^2$ and pointwise almost everywhere to $(h_1,h_2)\in H^1(\Omega)^2$.
By Fatou's lemma, we deduce that also $\phi_\epsilon(h_2-h_1)$ lies in $L^1(\Omega)$. Using the weak
lower semicontinuity of the norm, we obtain
\begin{equation}
\displaystyle E_{\epsilon}(h_1,h_2)\le \lim\inf_{k\rightarrow\infty} E_{\epsilon}(h_1^k,h_2^k)=\inf_{(\bar{h}_1,\bar{h}_2)\in V}E_{\epsilon}(\bar{h}_1,\bar{h}_2).\nonumber
\end{equation}
Consequently, $(h_1,h_2)$ is a minimizer and the integrability of $\phi_\epsilon(h_2-h_1)$ implies that  $h_2-h_1>0$ almost everywhere in $\Omega$.

Let now $d=1$ and $\Omega=(0,\,L)$. The estimate \rf{M1ref} implies
\begin{equation}
\int_\Omega|\px  (h_2^k-h_1^k)|^2\,dx\le  C_3.\nonumber
\end{equation}
Next, the boundedness of $\int \phi_\epsilon(h_2^k-h_1^k)\,dx$ and definition of $\phi_\epsilon$ imply a
uniform in $k$ bound on $||(h_2^k-h_1^k)^{-4}||_2$. Using this, one can estimate
\begin{align*}
\int_0^L|\partial_x((h_2^k-h_1^k)^{-3})|\,dx&=3\int_0^L\frac{|\partial_x(h_2^k-h_1^k)|}{(h_2^k-h_1^k)^4}\,dx\\
&\le 3||\partial_x(h_2^k-h_1^k)||_2\ ||(h_2^k-h_1^k) ^{-4}||_2\le C_4,
\end{align*}
where $C_4$ is constant. Owing to the continuous embedding of  $W^{1,3}(0,1)$ into $L^{\infty}(0,1)$ the strong positivity of $h_2-h_1$ follows.
This in turn implies the differentiability of the function $F_\epsilon(s):=E_{\epsilon}(h_1+s\varphi_1,h_2+s\varphi_2)$ considered with fixed $(\varphi_1,\varphi_2)\in H^1(0,L)^2$
for sufficiently small $s$. Since $(h_1,h_2)$ is a minimizer, by differentiation of $F_\epsilon(s)$ at $s=0$, we obtain that
\begin{equation}
\int_0^L(-\sigma\partial_{xx}h_1-\phi_\epsilon'(h_2-h_1))\varphi_1+(-\partial_{xx}h_2+\phi_\epsilon'(h_2-h_1))\varphi_2\,dx=0,\nonumber
\end{equation}
for all $(\varphi_1,\varphi_2)\in H^1(0,L)^2$ such that
\begin{equation}
\int_0^L\varphi_1\,dx=\int_0^L(\varphi_2-\varphi_1)\,dx=0.\nonumber
\end{equation}
Without this constraint, using Lagrangian multipliers and testing with general $(\psi_1,\psi_2)\in H^1(0,L)^2$ this yields
\begin{equation}
\int_0^L\Big[(-\sigma\partial_{xx}h_1-\partial_{xx}h_2)\psi_1 +(-\partial_{xx}h_2+\phi_\epsilon')\psi_2\Big] dx
-\frac{1}{L} \int_0^L\int_0^L\phi_\epsilon'\,dy\,\psi_2\,dx=0.\nonumber
\end{equation}
Standard elliptic regularity theory then implies that $(h_1,h_2)$ are smooth
solutions to  \rf{SS1} together with \rf{NBC} and \rf{l}.
\end{proof}\\

\begin{remark}
Note, that for Dirichlet boundary conditions we can proceed as follows:
Let us impose on system \rf{SS1} the Dirichlet boundary conditions
 \be
h_1(0)=h_1(L)=A,\ \ h_2(0)=h_2(L)=B,
\lb{DBC}
\ee
such that 
\be
\displaystyle B-A=\min_{x\in (0,L)} \bar{h}(x),
\lb{Eps}
\ee
where $\bar{h}$ is the Neumann solution to \rf{DE1} defined above.
In this case it follows again that $h_2-h_1=\bar{h}$. Consequently, $h_1$ and $h_2$
are given by \rf{RSS1} with constants $\la_1,\la_2,C,C_1$ determined uniquely by
\rf{DBC} and conservation of mass \rf{M1}-\rf{M2}.
Using \rf{Eps} and the asymptotics for $\bar{h}$ one obtains that the leading order of the solution \rf{RSS1} as
$\eps\go 0$ has now the form 
\begin{align}
h_1(x)&=\frac{1}{2}\frac{\la_1-\la_2}{\sigma}\left((x-\tfrac{L}{2})^2-s^2\right)+C_1,\ \
&&x\in\omega\nonumber\\
h_2(x)&=-\frac{1}{2}\la_1\left((x-\tfrac{L}{2})^2-s^2\right)+C_1,\ \
&&x\in\omega\nonumber\\
h_1(x)&=h_2(x)=-\frac{1}{2}\frac{\la_2}{\sigma+1}\left((x-\tfrac{L}{2})^2-s^2\right)+C_1,\ \
&&x\in(0,L)\setminus\omega,
\lb{DSI}
\end{align}
where $\omega=(L/2-s,L/2+s)$ and
\bes
s^2=\frac{2\sigma(\sigma+1)}{\left(\lambda_2-(\sigma+1)\lambda_1\right)^2}.
\ees
In contrast to the solution of \eqref{DSI} with Neumann conditions, solutions with Dirichlet
boundary conditions are not constant but quadratic in the ultra-thin layer $(0,L)\setminus\omega$.
\end{remark}

\section{Matched asymptotic solution and contact angles}

Note first that the system of equations for $h_1$ and $h_2$ \rf{steadyeqh1h2} is equivalent to following system for $h_1$ and $h$
\begin{subequations}\label{systemh}
\begin{align}
0 &= \p_x\left(-\sigma\p_{xx}h_1 - \phi_{\varepsilon}'(h)\right), \label{stath}\\
0 &= \p_x\left(-\frac{\sigma}{\sigma+1}\p_{xx}h +  \phi'_{\varepsilon}(h)\right).\label{stath1}
\end{align}
\end{subequations}
where we denote $\sigma=\sigma_1/\sigma_2$, see \cite{JMPW11} for a more detailed derivation. For our asymptotic analysis, this is convenient, since now for the variable $h=h_2-h_1$ we can distinguish the core droplet region, which we will call the
``outer region'' and the adjacent thin regions of thickness $\eps$, which we
call the ``inner region''. We will derive a sharp-interface limit using matched
asymptotic analysis in the limit as $\eps\to 0$. For this we first write the
equations in the form such that the intermolecular potential is small in the
core region and becomes order one in the adjacent thin regions. Using $(n,\ell)=(2,8)$ we define 
\be
\phi'_{\eps}(h) = \ds \frac{1}{\eps}\Phi'\left(\ds\frac{h}{\eps}\right)
\ee

\subsection{Stationary solution for $\boldsymbol h$}
As we will later show rigorously, we can assume that the droplet is (axi)symmetric and the profile a decreasing function, so that without loss of generality, the maximum of $h$ is at the origin of our coordinate system. Now consider the problem $h$ and $x\geq 0$
\begin{subequations}\label{problemhxxx}
\begin{align}
&0 = \p_x\left[\frac{\sigma}{\sigma+1}\p_{xx}h - \frac{1}{\eps}\Phi'\left(\ds\frac{h}{\eps}\right) \right], \\  \nn\\
&\lim_{x\to\infty} h = h_{\infty},\quad \lim_{x\to\infty} \px h = 0,\quad \lim_{x\to\infty} \pxx h = 0.
\end{align}
\end{subequations}
We can integrate this twice and use the conditions as $x\to\infty$ to fix the integration constants to obtain
\be\label{problemhx}
\p_x h = \sqrt{2\frac{\sigma +1}{\sigma}}\sqrt{\Phi\left(\frac{h}{\eps}\right)-\Phi\left(\frac{h_{\infty}}{\eps}\right)-\frac{1}{\eps}\Phi'\left(\frac{h_{\infty}}{\eps}\right)\left(h-h_{\infty}\right)}
\ee
The solution to this problem shows a steep decline in height towards $O(\eps)$ in an $\eps$-strip around $x=s$, where we would like to determine the apparent contact-angle. This can be obtained by writing the problem in so-called outer and inner coordinates, valid in the core and the adjacent thin regions, and matching as $\eps\to 0$. Interestingly, while it is easy to obtain the condition for the contact angle, it turns out that in order to carry out the complete matching consistently, we need to go up to second order in the matching, in order to account for the logarithmic switch-back terms, that come into play in this problem, see \cite{paco88} for a discussion of these terms.  

Note, that the coefficient $(\sigma+1)/\sigma$ can be removed by rescaling $x$ 
appropriately, leading to the classical problem of a droplet of height $h$ on a 
solid substrate.
Interestingly, to our knowledge for this problem the above mentioned logarithmic 
switch-back terms have not been noticed before.

\subsection*{Outer problem} 
The symmetry of the problem leads us to the condition that at the symmetry axis $x=0$ we have 
\be
\p_x h=0
\ee
It is also convenient to normalize the height such that $h(0)=1$. In this case
we obtain from \rf{problemhx} an algebraic equation for $h_{\infty}$ and $\eps$
that can be approximated as $\eps\to 0$. 
\be
0=\Phi\left(\frac{1}{\eps}\right)-\Phi\left(\frac{h_{\infty}}{\eps}\right)-\frac{1}{\eps}\Phi'\left(\frac{h_{\infty}}{\eps}\right)\left(1-h_{\infty}\right)
\ee
Solving this by making the ansatz for the asymptotic expansion for $h_{\infty}$ 
\be
h_{\infty}=\eps h_{\infty,0} + \eps^2 h_{\infty,1} + \eps^3 h_{\infty,2} + O(\eps^4)
\ee
we obtain
\be
h_{\infty,0}=1,\quad h_{\infty,1}=\frac{1}{16},\quad h_{\infty,2}=\frac{45}{512},\quad
\ee
Next, we assume that the asymptotic solution to the outer problem can be represented by the expansion
\be
h(x;\eps)=f_{0}(x) + \eps f_{1}(x) +   \eps^2 f_{2}(x) +    O(\eps^3).
\ee
The leading order outer problem then becomes 
\begin{subequations}\label{outerh0}
\begin{align}
&\p_x f_{0}=-\sqrt{\frac{3}{4}\frac{\sigma +1}{\sigma}\left(1-f_{0}\right)},\\
&f_{0}(0)=1
\end{align}
\end{subequations}
which has the solution 
\be
f_{0}(x)=-\frac{3}{16}\frac{\sigma + 1}{\sigma}x^2+1.\label{h0a}
\ee
Hence, the leading order outer solution will vanish as $x$ approaches the location 
\be
s = \frac{4}{\sqrt{3}}\sqrt{\frac{\sigma}{\sigma + 1}}\label{seq}
\ee  
However, the full solution does not vanish and will be obtained by matching to the solution of the ``inner'' problem near $s$. 
We will find that in order to complete the solution, we will need to solve the expansion up-to second order. 
We find for $f_{1}$ and $f_{2}$ 
\begin{subequations}\label{outer1}
\begin{align}
&\p_x f_{1}=\frac{f_{1}}{x}-\frac{3}{16}\frac{\sigma +1}{\sigma}\, x,\\
&f_{1}(0)=0
\end{align}
\end{subequations}
and 
\begin{subequations}\label{outer2}
\begin{align}
&\p_x f_{2}=2\frac{f_{2}}{x}+\frac{1}{x}\left(\frac{8}{3}\frac{1}{f_{0}^2}+\frac{1}{16}\left(f_{0}-1\right)-\frac{2}{3}\left(f_{0}+3\right)+2f_{1}\right)+\frac{3}{8}\frac{\sigma + 1}{\sigma}\, x \\
&f_{2}(0)=0
\end{align}
\end{subequations}

\subsection*{Inner problem}

The solution of the inner problem lives in an $\eps$ neighborhood of $x=s$ and
extends towards $x\to +\infty$. It will be matched to the outer problem in the
other direction. Hence we introduce the inner variables $v(z)$ and independent 
variable $z$ via 
\begin{equation}\label{inner var}
h(x)=\varepsilon\,v(z; \eps) \quad\mbox{and}\quad  x=s+\varepsilon z.
\end{equation}
Rewriting the problem \rf{problemhx} in these coordinates and making the ansatz 
\be
v(z;\eps)=v_0(z) + \eps v_1(z) + O(\eps^2)
\ee
we find to leading order the problem 
\be 
\p_z v_0 = -\sqrt{2\frac{\sigma +1}{\sigma}}\sqrt{-\frac{1}{2v_0^2}+\frac{1}{8v_0^8}+\frac34}
\ee
and to $O(\eps)$ the problem 
\be
\p_z v_1=-\sqrt{2\frac{\sigma+1}{\sigma}}\sqrt{-\frac{1}{2v_0^2}+\frac{1}{8v_0^8}+\frac34}\,\,\frac{{3}v_0^9 (v_0-1) + 8v_1(1-v_0^6)}{2v_0(4v_0^6-3v_0^8-1)}
\ee
We solve and match in the inner coordinates and obtain by expanding $v$ at $z=-\infty$ 
\bea\label{expv}
v_0+\eps v_1 &=& -\frac{\sqrt{3}}{2}\sqrt{\frac{\sigma + 1}{\sigma}}\, z + a_1 -\frac{4}{9}\sqrt{3\frac{\sigma}{\sigma+1}}\,\frac{1}{z} + \cdots\nn\\
&& +\eps\left(-\frac{3}{16}\frac{\sigma +1}{\sigma}\, z^2 + \frac{\sqrt{3}}{4}\frac{\sigma + 1}{\sigma}(a_1-1)\, z - \ln(-z) \right.\nn\\
&&\hspace*{2cm} \left.+ C + \frac16 +\frac23\sqrt{3\frac{\sigma}{\sigma+1}}(a_1+1)\frac1z+\cdots\right).
\eea
For the corresponding outer expansion we have 
\bea\label{exph}
\frac{f_0+\eps f_1 + \eps^2 f_2}{\eps} &=& -\frac{\sqrt{3}}{2}\sqrt{\frac{\sigma + 1}{\sigma}}\, z - 1 -\frac{4}{9}\sqrt{3\frac{\sigma}{\sigma+1}}\,\frac{1}{z} + \cdots\nn\\
&& +\eps\left(-\frac{3}{16}\frac{\sigma +1}{\sigma}\, z^2 - \frac{\sqrt{3}}{2}\frac{\sigma + 1}{\sigma}\, z - \ln(-z) \right.\nn\\
&&\hspace*{1cm} \left. -\ln\left(\frac{\sqrt{3}}{8}\sqrt{\frac{\sigma + 1}{\sigma}}\right) -\frac{19}{96} - \ln(\eps)+\cdots\right).
\eea

We note that all of the terms in the first row of \rf{expv} and \rf{exph} match provided $a_1=-1$. The terms in the second row match provided 
\be
C=-\ln(\eps) -\frac{35}{96}-\ln\left(\frac{\sqrt{3}}{8}\sqrt{\frac{\sigma + 1}{\sigma}} \right)
\ee
where we note the appearance of a so-called ``logarithmic switch-back'' term $\ln(\eps)$. 
Hence, the composite solution is 
\bea\label{compositesol}
\bar h&=& \eps\left[v_0\left(\frac{x-s}{\eps}\right) +\frac{\sqrt{3}}{2}\sqrt{\frac{\sigma + 1}{\sigma}}\,\frac{x-s}{\eps}\right] \\
&&\hspace*{-0.5cm}+\eps^2\left[ v_1\left(\frac{x-s}{\eps}\right)+\frac{3}{16}\frac{\sigma}{\sigma+1}\left(\frac{x-s}{\eps}\right)^2+\frac{\sqrt{3}}{2}\frac{\sigma + 1}{\sigma}\,\frac{x-s}{\eps}\right]\nn\\
&&\hspace*{-0.5cm}+f_0(x)+\eps\Big[f_1(x) +1\Big]\nn\\
&&\hspace*{-0.5cm}+\eps^2\left[f_2(x)+\frac{4}{9}\sqrt{\frac{3\sigma}{\sigma+1}}\,\frac{1}{x-s}+\ln\left(s-x\right)+\frac{19}{96}+\ln\left(\frac{\sqrt{3}}{8}\sqrt{\frac{\sigma + 1}{\sigma}}\right)\right]\nn
\eea
with $s$ given in equation \rf{seq} and for $x<s$. For $x\ge s$ only the inner expansion $\bar{h}=\eps v_0 + \eps^2 v_1$ remains. 

\begin{figure}
\centering
\includegraphics[width=0.9\textwidth]{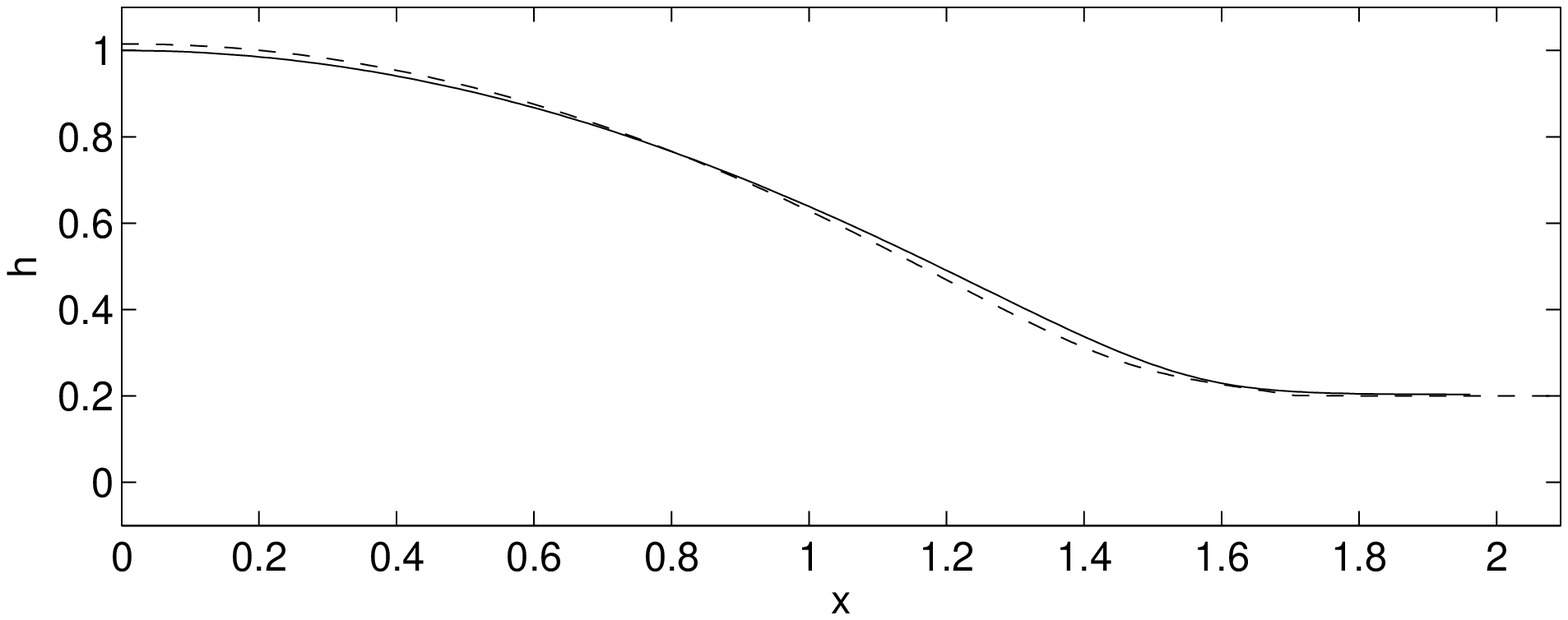}

\includegraphics[width=0.9\textwidth]{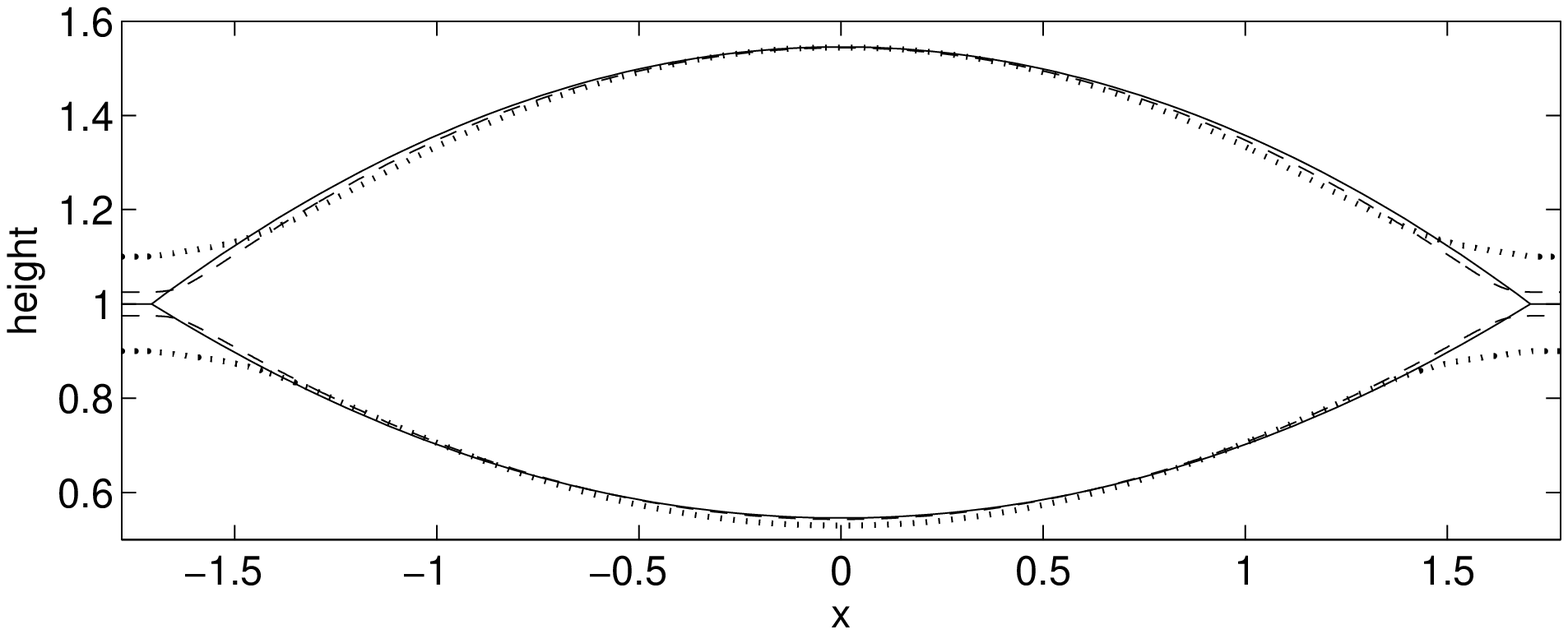}

\caption{Top: Comparison of the composite solution $h_c$ (dashed curve) with the numerical
solution of \eqref{problemhxxx} (solid curve), for $\varepsilon=0.2$ and $\sigma=1.2$.
Bottom: Solutions for $h_1$ and $h_2$ reconstructed from the composite solution $\bar
h$ for $\sigma=1.2$ and $\varepsilon=0.05$ (dashed curve) and $\varepsilon=0.2$
(dotted curve) and the solution to the sharp interface model \eqref{sharpinterace} (solid curve).}
\label{comparisonasymptotics}
\end{figure}

\subsection*{Stationary solution for $h_1$ and $h_2$}

To complete the solution, we determine the solution to the liquid-liquid
interface $h_1$ simply by adding equations \rf{stath} and \rf{stath1}. Then we
integrate thrice and use the far-field conditions
$\pxx h, \px
h, \pxx h_1, \px h_1 \to 0$, $h\to h_{\infty}$ and $h_1\to d$ as $x\to\pm\infty$ to fix the constants. This results in $h_1$ being
\be
h_1 = - \frac{1}{\sigma + 1}\left(h-h_{\infty}\right) + d\label{h1sol},
\ee
and equivalently $h_2$ is
\be
h_2=\frac{\sigma}{\sigma + 1}\left(h-h_{\infty}\right) +d+h_\infty\label{h2sol}.
\ee
Here $d$ denotes the height $h_1$ as $x\to\pm\infty$ and we assume $d$ to be large enough so that $h_1$ never becomes negative. Also note that for other boundary conditions, such as e.g. Dirichlet conditions mentioned in the previous section, further contributions will arise. Unlike the solutions for droplets on solid substrates, here new families of profiles for the ultra-thin film connecting a droplet to the boundaries or to other droplets arise. 

In figure \ref{comparisonasymptotics} we compare our asymptotic solution with the numerical solution. Observe that the $O(\eps)$ solution already gives an excellent approximation of the exact solution for $\eps=0.2$, where the exact solution is approximated by the higher-order numerical solution of the boundary value problem. This suggests that a sharp-interface model should also be a good approximation of the full model. The sharp-interface model for $h$ is simply the leading order outer problem for $h$, but now with boundary conditions, that result from the leading order matching. Hence, we obtain 
\begin{subequations}\label{sharpinterace}
\begin{align}
&0=\pxxx f_0\\
\intertext{with boundary conditions}
&f_0=0 \quad\mbox{and}\quad \px f_0= \p_{z}v_0= - \sqrt{-2\frac{\sigma+1}{\sigma}\Phi(1)}  \quad\mbox{at}\quad x=\pm s
\end{align}
\end{subequations}
where the contact angle has been determined via matching. Equivalently, for the half-droplet, by imposing symmetry we have the boundary conditions
\be
f_0(s)=0, \quad \px f_0(s)= - \sqrt{-2\frac{\sigma+1}{\sigma}\Phi(1)} \quad\mbox{and}\quad  \px f_0(0)= 0
\ee

In the following section we show how the sharp-interface model is obtained via
$\Gamma$-convergence and also proof existence and uniqueness of its solutions.

\section{Sharp-interface limit via $\boldsymbol \Gamma$-convergence}

In this section we investigate properties of stationary solutions of the
sharp-interface two-layer model. Such model can be obtained
by considering the limiting problem $\eps \to 0$ in the
framework of $\Gamma$-convergence. For one-layer systems corresponding
minimizers are known as mesoscopic droplets \cite{otto2007coarsening}. 
In this approach equilibrium contact angles can be directly deduced from the 
Euler-Lagrange equation of the resulting $\Gamma$-limit energy. On the other 
hand showing an equi-coerciveness property we have that the sequence of
minimizers of $E_\eps$ converges to a minimizer of the $\Gamma$-limit energy $E_\infty$.

For boundary conditions $h_1=h_2$ and certain domains we show that solutions 
of the minimization problem $\min_{(h_1,h_2)}E_\infty$
exist and are unique up-to translations. The main technique used here is
the symmetric-rearrangement, see e.g. \cite{leoni2009first, lieb2001analysis}.

For the section to come we consider energies such as in \eqref{DE}. For later 
convenience we  define $\potG(h)=\bigl(\pot(h)-\Phi(1)\bigr)/\phim$ 
where as before $\PotG(h)=\potG(h/\eps)$ is independent of $\eps$. The shift
by $\phim$ has the advantage of working with a non-negative energy without 
changing the Euler-Lagrange equations.

With these definitions consider the following family of minimization problems. 
For $\Omega\subset\mathbb{R}^n$ bounded with Lipschitz boundary and
$m_1,m_2>0$ given and $\eps>0$ we look for minimizers 
of $E_\eps:\Hm\to \mathbb{R}^\infty$ defined as 
\begin{align}
\label{eqn:gammaenergy}
E_\eps(h_1,h_2)=
        \intd \frac{\sigma}{2}|\nabla h_1|^2 + \frac{1}{2}|\nabla
h_2|^2 + \,\phim\,\potG(h_2-h_1) 
\end{align}
with $\sigma>0$ and $\potG(h)=\PotG\bigl(h/\eps\bigr)$ as $\eps\to 0$. Note that
nonnegativity $h_1>0$ is not enforced because otherwise extra terms in $E_\eps$
are required.  

\subsection{$\boldsymbol \Gamma$-convergence}

For a given domain $\Omega\subset\mathbb{R}^n$ bounded with Lipschitz boundary
define the space of admissible interfaces as before in \eqref{eqn:spacedef} by
\begin{equation}
\Hm:=\left\{(h_1,h_2)\in H^1(\Omega)^2:
	m_1=\intd h_1, \
	m_2=\intd (h_2-h_1), \
	h_2\ge h_1 
\right\},
\end{equation}
and in general let $\PotG:\mathbb{R}\to\mathbb{R}$ be a function with the properties\\

\begin{enumerate}
\item[(i)] $\PotG\ge0$ \,and\, $\PotG(h)=0 \Leftrightarrow h=1$.
\item[(ii)]\label{prop:psi_infty} $\PotG(h)\xrightarrow{h\to +\infty} 1$ and
$\PotG(h)\le 1 $ for $h>1$.\\
\end{enumerate}

We want to investigate the family of minimization problems \eqref{eqn:gammaenergy}.
First we note that the sequence $E_\eps$ is equi-coercive in the weak topology
of $H^1(\Omega)^2$.
This is a simple consequence of
\begin{align*}
E_\eps(h_1,h_2)\ge c ( \|\nabla h_1\|_{H^1}^2+\|\nabla  h_2\|_{H^1}^2),
\end{align*}
which holds for all $(h_1,h_2)\in \Hm$. Together with the $\Gamma$-convergence
the equi-coercivity implies the following abstract convergence result. We know
that any sequence $\{h_{1,n},h_{2,n}\}$ of minimizers to the energies
$E_{\eps_n}$ has a weakly converging subsequence $(h_{1,n},h_{2,n})\wto
(h_{1}^*,h_{2}^*)$. Furthermore the limit $(h_{1}^*,h_{2}^*)$ is a minimizer of
the $\Gamma$-limit energy $E_\infty$. This relates minimizers of $E_\eps$ to
minimizers of the $\Gamma$-limit $E_{\infty}$.

Now we investigate the $\Gamma$-limit of \eqref{eqn:gammaenergy} in the topology
of weak convergence in the space $H^1(\Omega)^2$. 
We recall the definition of $\Gamma$-convergence , see also
\cite{braides2002gamma,dal1993introduction}. 
\\

\begin{definition}
We say that a sequence $E_\eps:X\to\mathbb{R}^\infty$ $\Gamma$-converges in $X$
in the weak topology to $E_\infty:X\to\mathbb{R}^\infty$ if for all $x\in X$ we have
\begin{itemize}
 \item[(i)]  (\textit{lim-inf inequality}) 
For every sequence $\{x_\eps\}\subset X$ weakly converging to $x$ there holds
\[
 E_\infty(x)\le \liminf_\eps E_\eps(x_\eps) .
\]
 \item[(ii)]  (\textit{lim-sup inequality}) 
There exists a sequence $\{x_\eps\}\subset X$ weakly converging to $x$ 
such that
\[
 E_\infty(x) \ge \limsup_\eps E_\eps(x_\eps).
\]
\end{itemize}
The function $E_\infty$ is called the $\Gamma$-limit of $(E_\eps)$, and we write
$E_\infty\!=\!\Gamma\text{-}\lim_\eps E_\eps.$
\end{definition}
\\

The key proposition to compute the overall $\Gamma$-limit is to consider the 
$\Gamma$-limit of the potential separately. Here we use that weak convergence 
in $H^1$ implies strong convergence in $L^2$ and the right continuity of $q\mapsto\int \chi\{h>q\}$ 
for any given $h\in H^1$.
\\

\begin{proposition}
\label{gammaconvproposition}
Consider the functional $F_\eps:H^1(\Omega)\to\mathbb{R}^\infty$ defined as 
\[
F_\eps(h)=\begin{cases}\intd \potG(h) & h\in X_m, \\
           \infty & \text{otherwise},
	  \end{cases}
\]
where 
\[
 X_m=\left\{h\in H^1(\Omega): h\ge0, \intd h=m\right\}.
\]
Then 
\[
\Gamma\text{-}\lim_{\eps}F_\eps(h)=F_\infty(h)=\begin{cases}\intd \chi\{h>0\} &
h\in X_m , \\
           \infty & \text{otherwise},
	  \end{cases}
\]
with $\chi$ being the characteristic function.
\end{proposition}
\begin{proof} 
Consider an arbitrary sequence $\eps_n \to 0$ and $h\in X_m$.
\\
(i) (\textit{lim-inf condition}) 
\\
Let $\{h_n\} \subset X_m$ such that $h_n \wto h$ weakly in $H^1(\Omega)$, then
$h_n\to h$ strongly in $L^2(\Omega)$. Choose $\delta_n\to 0$ such that 
${\eps_n/\delta_n}\to 0$ as $n\to \infty$ which immediately gives 
\begin{align}\label{eq_54}
\liminf_n \intd W_{\eps_n}(h_n) 
&\ge \liminf_n\int_{\{h_n>\delta_n\}} W_{\eps_n}(h_n)\,\,
{=}\,\,\liminf_n \intd \chi\{h_n>\delta_n\} .
\end{align}
Next we want to use $\int\chi\{h>0\}\le\liminf \intd \chi\{h_n>\delta_n\}$. 
Conversely assume
\begin{align}\label{eq_55}
\liminf_n \intd \chi\{h_n>\delta_n\}<\int\chi\{h>0\}.
\end{align}
Then employing right-continuity of $s\mapsto \intd \chi\{h>s\}$ 
(see \cite{leoni2009first}, Proposition 6.1)
there must exist some $\delta,\bar\delta>0$
such that 
\begin{align*}
 0&<\liminf_n \intd \chi\{h>0\}-\chi\{h_n>\delta_n\}-\delta
  \le \liminf_n\intd\chi\{h>\bar\delta\}-\chi\{h_n>\delta_n\}
,\\
  &\le \liminf_n\intd\chi\{h>\bar\delta \mathrm{\,\,\&\,\,} h_n<\delta_n\}
  \le\left(\frac{2}{\bar\delta}\right)^2\liminf_n\intd |h-h_n|^2= 0,
\end{align*}
where we used Chebyshev's inequality. This is a contradiction and thus by the previous assertions
\begin{align*}
\liminf_n \intd W_{\eps_n}(h_n) 
\ge \intd \chi\{h>0\}.
\end{align*}
\\
(ii) (\textit{lim-sup condition}) 
\\
Define a recovery sequence by
$h_n = \alpha_n h + \eps_n$ 
where
$
\alpha_n = {(m-\eps_n|\Omega|)}/{m}
$.
Then $h_n\in X_m$ and $ h_n\to h$ even strongly in $H^1(\Omega)$ and the
following estimate holds
\begin{align*}
\limsup_n \intd \potG_n(h_n)
	&=\limsup_n\Big(
		 \int_{\{h>0\}} \PotG\left(1+\frac{\alpha_n}{\eps_n}h\right)
		  + \int_{\{h=0\}} \PotG\left(1+\frac{\alpha_n}{\eps_n}h\right)
	\Big)
\, ,
\\ 	&\le 
	\limsup_n \intd \chi\{h>0\}
	+\int_{\{h=0\}} \PotG\left(1\right)
	=
	\intd \chi\{h>0\}
\, ,
\end{align*}
where we used that $\PotG(s)\le 1$ for $s>1$ and $\PotG(1)=0$.
\end{proof}
\\

To prove the $\Gamma$-convergence to the sharp-interface model we can exploit 
the property that the behavior of gradient terms can be easily controlled.\\

\begin{theorem}
\label{thm:gammalimit}
For the family of energies \eqref{eqn:gammaenergy} the $\Gamma$-limit is 
\[
E_\infty(h_1,h_2)=
        \intd \frac{\sigma}{2}|\nabla h_1|^2 + \frac{1}{2}|\nabla
h_2|^2 + \phim\,\chi\{h_2>h_1\} 
\]
\end{theorem}

\begin{proof}
The gradient terms in $E_\eps$ are weakly lower semicontinuous with respect to weak
convergence in $H^1(\Omega)^2$. 
Together with Proposition~\ref{gammaconvproposition} this gives the
\textit{lim-inf inequality}.
On the other hand the gradient terms are continuous with respect to strong
convergence in $H^1(\Omega)^2$. Choosing the recovery sequence as in the proof of
Proposition~\ref{gammaconvproposition} 
one gets the desired \textit{lim-sup inequality}.
\end{proof}
\\

Now we want to deduce necessary conditions for minimizers of $E_\infty$. We are
especially
interested in conditions at the points where the two-phase domain
meets the one-phase domain.
One problem is that one cannot immediately compute the Euler-Lagrange equations
($L^2$-gradient) of the sharp-interface 
energy functional $E_\infty$. 
This is due to $\intd \chi\{h_2>h_1\}$ being 
only lower semicontinuous in the strong $H^1(\Omega)$ topology, but not
continuous nor differentiable.
In fact directional derivatives of this part of the energy will almost surely
be zero or infinite. 
Therefore we compute the directional derivative of $E_\infty$ in another 
topology as follows: For ease of notation introduce 
\begin{equation}
 \omega=\{x\in\Omega:h_2>h_1\}
\end{equation}
and restrictions of $h_1$ and $h_2$ to $\omega$ and $\Omega\setminus\omega$ are called 
\[
h_1:=h_1|_\omega,\quad h_2:=h_2|_\omega,\quad
h:=h_1|_{\Omega\setminus\omega}=h_2|_{\Omega\setminus\omega},
\]
with boundary condition $h_1=h_2=h$ on $\partial\omega$. We will now vary $h_1,h_2$ and $h$ but also
$\omega$. The formal calculation is
restricted to smooth $h_1,h_2,h$ and $\omega$.
Using these we can rewrite the energy using the restrictions as
\begin{align}
\label{eqn:einfrestriction}
E_\infty(h_1,h_2,h,\omega)
  &=\int_\omega  \left[
      \frac{\sigma}{2}|\nabla h_1|^2 + 
      \frac{1}{2}|\nabla h_2|^2 + \phim\right]+
    \int_{\Omega\setminus\omega} \frac{\sigma'}{2}|\nabla h|^2,
\end{align}
where we define $\sigma':=1+\sigma$. 
A perturbation of $\tau\mapsto\bigl(h_1(\tau),h_2(\tau),h(\tau),\omega(\tau)\bigr)$
can be pa\-ram\-e\-trised
using a diffeomorphism $\psi(\circ,\tau):\Omega\to\Omega$ with the property that
$\omega(\tau)=\{\psi(x,\tau):x\in\omega(0)\}$. In the same spirit define $\bar h_1(x_0,\tau)=h_1(\psi(x_0,\tau),\tau)$ as the
pullback of $h_1$ by $\psi$ and similarly $\bar{h}_2$ and $\bar{h}$. 
The boundary conditions
$\partial_\tau{\bar{h}_1}=\partial_\tau{\bar{h}_2}=\partial_\tau{\bar{h}}$ on
$\partial\omega(\tau)$ 
translate into
\begin{align}
\label{eqn:bcgammaconv}\dot{h}_1+\dot{\psi}\cdot \nabla h_1
=\dot{h}_2+\dot{\psi}\cdot \nabla h_2
=\dot{h}+\dot{\psi}\cdot \nabla h=:\dot{\xi}.
\end{align}
Here we use the notation $\dot\psi:=\partial_\tau\psi$, $\dot{h}_1:=\partial_\tau
h_1$, $\dot{h}_2:=\partial_\tau h_2$ and $\dot{h}:=\partial_\tau h$. 
Then using Reynolds transport theorem we get 
\begin{align*}
\frac{d}{d\tau}E_\infty(h_1,h_2,h,\omega)
  =&\int_\omega \left(
	  {\sigma}\nabla h_1\nabla \dot{h}_1+
	  \nabla h_2\nabla \dot{h}_2\right)
    +\int_{\Omega\setminus\omega} 
	{\sigma'}\nabla h\nabla \dot{h}
\\&\quad
    +\int_{\partial\omega} \left[\frac{\sigma}{2}|\nabla h_1|^2 +
\frac{1}{2}|\nabla h_2|^2-\frac{\sigma'}{2}|\nabla
h|^2+\phim\right](n\cdot\dot{\psi}).
\end{align*}
Applying integration-by-parts and boundary conditions \eqref{eqn:bcgammaconv}
yields in one and two spatial dimensions, i.e. $\omega\subset\Real$ and $\omega\subset\Real^2$ the 
directional derivative
\begin{align}%
\nonumber
0=&\frac{d}{d\tau}\Big(E_\infty+\lambda_2\int h_1+\lambda_1 \int
(h_2-h_1)\Big),\\
\nonumber=&-\int_\omega \bigl[\sigma \Delta h_1+\lambda_2-\lambda_1\bigr]\dot
{h}_1+\bigl[ \Delta h_2+\lambda_1\bigr]\dot
{h}_2\,\,-\int_{\Omega\setminus\omega}\bigl[\sigma' \Delta
h+\lambda_2\bigr]\dot{h}\\
\nonumber&\,\,+\int_{\partial\omega}\left[-\frac{\sigma}{2}(\nabla
h_1)^2-\frac{1}{2}(\nabla h_2)^2+\frac{\sigma'}{2}(\nabla
h)^2+\phim+\eta^2\bigl(1+\sigma-\sigma'\bigr)\right]
(n\cdot\dot\psi)\\
\label{eqn:variationofenergy2}
&\,\,+\int_{\partial\omega}\left[\sigma(n\cdot\nabla h_1)+(n\cdot\nabla
h_2)-\sigma'(n\cdot\nabla h)\right]\bigl(\dot{\xi}+(t\cdot\dot\psi)\bigr).
\end{align}%
where $\eta^2:=(t\cdot\nabla h_1)^2\equiv(t\cdot\nabla h_2)^2\equiv(t\cdot\nabla
h)^2$. 
The expressions inside square brackets have to vanish independently,
since the perturbations $(\dot{h}_1,\dot{h}_2,\dot{h},\dot{\xi},\dot{\psi})$ are
independent. 
\\

\begin{remark}
Above we added Lagrange multipliers $\lambda_1,\lambda_2$ to take care of the
mass conservation.
In one dimension there is no tangential  contribution and hence $\eta\equiv 0$,
whereas in two dimensions the contribution with $\eta^2$ vanishes due to definition 
$\sigma'=\sigma+1$. However, if $\sigma'$ could be choosen independently of $\sigma$,
there would be an extra contribution in that case.
\end{remark}

\subsection{Existence and uniqueness of solutions} 

In this part we consider the sharp interface energy derived by $\Gamma$-convergence and study its minimizers. The idea of the proof is to show that for a minimizer the support of
\begin{equation}\label{eqn:zdefine}
\zdiff:=h_2-h_1
\end{equation} 
is a ball contained in $\Omega$, on which the solutions can be computed explicitly. The minimization itself is performed with masses $\mathbf{m}=(m_1,m_2)$ held fixed. Further extensions of our proof and properties of the solutions are discussed in the end of this section.
\\
 
\begin{definition}
Let $A\subset\R^n$ a Borel set of finite Lebesgue measure, then 
the symmetric rearrangement of the set $A$ is defined by $A^*=\mathcal{B}_s(0)$ with $s$ 
such that $\mu(A)=\mu(A^*)$.
The symmetric decreasing rearrangement of
the characteristic function is $(\chi_A)^*=\chi_{A^*}$. Now let $f:\R^n\to\R$ a Borel measureable function
vanishing at infinity, then define the symmetric-decreasing rearrangement of $f$ by 
\[
 f^*(x)=\int_0^\infty \chi^*_{\{f>s\}}(x)\,{\rm d}s.
\]
\end{definition}
\\

\begin{theorem}{(Minimizer of sharp interface energy)}
\label{thm:minsharpinterf}
Let $\Omega=\mathcal{B}_R(0)$ and $X=\{(h_1,h_2)\in\Hm(\Omega):(h_1-h_2)|_{\partial\Omega}=0\}$ and energy 
\begin{align}
\label{eqn:classicalenergy}
E(h_1,h_2)
 &:=\int_{\dom} \frac{\sigma}{2} |\nabla h_1|^2 +\frac{1}{2} |\nabla h_2|^2 \ + \phim\,\chi_{\{h_2>h_1\}} \intdx.
\end{align}
Then using $\zeta(x):=\alpha (s^2-|x|^2)^+$ minimizers of $E$ 
with mass $(m_1,m_2)$ are
\begin{align}
\label{eqn:sisolform}
 h_2=\frac{\sigma}{\sigma+1}\zeta(x-x_0)+h, \qquad h_1(x)=h_2-\zeta(x-x_0),
\end{align}
with constant $x_0\in\Omega$ and $r,\alpha,h\in\mathbb{R}$. Prescribing the mass $(m_1,m_2)$ fixes $r$ and $h$, whereas
$\alpha$ is fixed by the contact angle (Neumann triangle) 
\begin{equation}
\label{eqn:contactangle}
{\sigma}(\nabla h_1)^2+{}(\nabla h_2)^2=2\phim, \qquad \text{ at }|x|=r.
\end{equation}
For large masses $m_2$ \eqref{eqn:contactangle} is not required and we get $r=R$, $x_0=0$ in \eqref{eqn:sisolform}.
\end{theorem}
\\

\begin{proof}
Using similar ideas as in \cite{otto2007coarsening}, we proceed as follows:

Symmetry: For given $(h_1,h_2)\in X$ let $\zdiff=(h_2-h_1)\in H^1_0(\dom)$ as in \eqref{eqn:zdefine}, non-negative  let 
\[
 \lambda:=\frac{\|\nabla h_2\|_{\Ldom}}{\|\nabla \zdiff\|_{\Ldom}}\ge 0.
\]
Using Cauchy-Schwarz inequality in $L^2(\Omega)$ the following energy estimate holds
\begin{align*}
 E(h_1,h_2)
&\ge
  (\sigma+1) \|\nabla h_2\|^2_{\Ldom} +\sigma \|\nabla \zdiff\|^2_{\Ldom} 
      - 2\sigma \|\nabla h_2\|_{\Ldom}\ \|\nabla \zdiff\|_{\Ldom} + \phim\mu(\{\zdiff>0\}),
\\
&=\|\nabla \zdiff\|^2_{\Ldom}\left(\lambda^2+\sigma(1-\lambda)^2\right) + \phim\mu(\{\zdiff>0\}).
\end{align*}
Minimizing with respect to $\lambda$ gives the lower bound
\begin{align*}
 E(h_1,h_2)
&
\ge
\frac{\sigma}{\sigma+1}\|\nabla \zdiff\|^2_{\Ldom}
    + \phim\mu(\{\zdiff>0\})
\end{align*}
which is attained only if 
$
\lambda={\sigma}/{(\sigma+1)}
$
and $\nabla h_2$ is a multiple of $\nabla \zdiff$. Now let $\zdiff^*$ be the symmetric-decreasing rearrangement of $\zdiff$, then 
by virtue of the P\'{o}lya-Szeg\H{o} inequality
\begin{align*}
\int_{\mathbb{R}^n}\zdiff^*\intdx = \int_{\mathbb{R}^n} \zdiff\intdx, \qquad\qquad \|\nabla \zdiff^*\|_{L^2(\mathbb{R}^n)}\le\|\nabla \zdiff\|_{L^2(\mathbb{R}^n)}.
\end{align*}
We have the freedom to translate $\zdiff^*$ as long its support is contained in $\Omega$.
Equality holds only if $\zdiff$ is already symmetric-decreasing \cite{lieb2001analysis}.
Now assume that $h$ is not symmetric decreasing, or $\nabla h \neq \sigma/(\sigma+1)\nabla h_2$.
Then we can reduce the energy by defining $h^*_1$ and $h^*_2$ by
\begin{align*}
h^*_2(x)&:=\frac{\sigma}{\sigma+1} \zdiff^*(x-x_0)+h, \qquad h^*_1(x):=h^*_2(x)-{\zdiff^*}(x-x_0),
\end{align*}
and have $\nabla h^*_2=\lambda \nabla \zdiff^*$ and $\mu\{\zdiff>0\}=\mu\{\zdiff^*>0\}$ so that
\begin{align*}
E(h_1,h_2)>\frac{\sigma}{\sigma+1}\|\nabla {\zdiff^*}\|^2_{\Ldom}+\phim\mu(\{{\zdiff^*}>0\})=E(h^*_1,h^*_2).
\end{align*}
Note that by definition $\{\zdiff^*>0\}=\mathcal{B}_s(0)=:\omega^*$ with $s$ such that $\mu(\omega^*)=\mu(\zdiff>0)$. To check $\zeta(x)=\alpha (s^2-|x|^2)^+$ 
is now analogous to \cite{otto2007coarsening}. One has to solve the Euler-Lagrange equation for the first variation of $E$ given in \eqref{eqn:variationofenergy2} using
standard methods.
\end{proof}
\\

\begin{corollary}
\label{cor:extendedenergy}
Let $X$ be as before and the sharp interface energy
\begin{align}
\label{eqn:otherenergy}
E_\infty(h_1,h_2,h,\omega)
 &:=\int_{\omega} \left(\frac{\sigma}{2} |\nabla h_1|^2 +\frac{1}{2} |\nabla h_2|^2 \ + \phim\right)\!\!\intdx + \int_{\Omega\setminus\omega} \frac{\sigma'}{2}|\nabla h|^2\intdx 
\end{align}
as in \eqref{eqn:einfrestriction} for $\sigma'>0$ arbitrary.
Then the minimizers of \eqref{eqn:classicalenergy} and \eqref{eqn:otherenergy} in $X$ are identical. 
\end{corollary}
\\

\begin{proof}
Since
we have $\zdiff=0$ on $\partial\omega$ the estimates of the previous proof are valid if the domain of 
integration is restricted to $\omega$. By construction we have 
$\|\nabla h\|_{L^2(\Omega\setminus\omega)}^2 \ge \|\nabla h^*\|_{L^2(\Omega\setminus\omega^*)}^2=0$. 
\end{proof}
\\

\begin{remark}
\label{rem:choiceofconstants}
Using the abbreviation $c=\phim(\sigma+1)/\sigma$ we can easily compute the parameters
$r$ and $\alpha$ from the previous theorem and get 
\begin{align*}
&&s_{1d}&=\left(\frac{9m_2^2}{8c}\right)^{1/4}\!\!\!,\quad
& \alpha_{1d}&=\left(\frac{2c^3}{9m_2^2}\right)^{1/4}\!\!\!,&&\\
&&s_{2d}&=\left(\frac{8m_2^2}{\pi^2c}\right)^{1/6}\!\!\!, \qquad
& \alpha_{2d}&=\left(\frac{\pi c^2}{2m_2}\right)^{1/3}\!\!\!.&&
\end{align*}
in one and two spatial dimension respectively. The contact angles are then
\begin{align}
\mathbf{n}\cdot\left.\nabla h_1 \right|_{s_-}=\pm\frac{\sqrt{2c}}{1+\sigma} && \mathbf{n}\cdot\left.\nabla h_2 \right|_{s_-}=\mp\sqrt{2c}\frac{\sigma}{1+\sigma}.
\end{align}
and which actually holds in any spatial dimension. We also have 
\begin{align}
\label{eqn:gammaconvcontangle}
\mathbf{n}\cdot\nabla(h_2-h_1)=\sqrt{2c}=\sqrt{-\frac{2(\sigma+1)}{\sigma}\Pot(1)} 
\end{align}
which can be compared with the 
appropriate boundary condition in \eqref{sharpinterace} from the matched asymptotic expansion.
\end{remark}

\section*{Discussion and Outlook}

We considered stationary solutions of systems of coupled thin-film equations for two-layer liquid films. After proving existence of stationary droplet solutions, we used matched asymptotic analysis to derive a corresponding sharp-interface model in the limit when the thickness of the ultra-thin film in the dewetted region $\eps\to 0$, which then yields the equilibrium Neumann angles. We point out that our asymptotic analysis requires the inclusion of logarithmic switch-back terms for the asymptotic droplet solution, which should in principle also be needed for the limiting case of droplet solutions on solid substrates. 

We then proved the existence and uniqueness of the sharp-interface model using  the variational structure of the equations allowing us to formulate the problem as a minimization problem, for which we can study the limit $\eps\to 0$ via  $\Gamma$-convergence. In one spatial dimension on an interval both sharp-interface models are equivalent. In particular the contact angle of $h$ from the matched asymptotic analysis in \eqref{sharpinterace} is the same as the one from the $\Gamma$-convergence in \eqref{eqn:gammaconvcontangle}. Since the recovery of $h_1, h_2$ from $h$ in both cases works via {(\ref{h1sol}, \ref{h2sol})} or \eqref{eqn:sisolform}, the second contact angle agrees as well. We note that dimensions $d>1$ one has to prove that the shape of the domain $\{h_2-h_1>0\}$ is a ball of a certain radius. Using symmetric decreasing rearrangement this property, and thereby existence and uniqueness of minimizers, could be proved. 

We expect that, as for thin films on solid substrate, the techniques of matched asymptotic analysis can  be extended to the dynamic time-dependent problem. In particular, the derivation and study of the time-dependent sharp-interface model will also support the understanding of the energetic structure of the system of the coupled thin film model and should still be valid in the time-dependent problem, e.g. in the gradient flow structure of a sharp-interface model. This will be important in the study of dewetting regimes, dewetting rates and the stability properties of the evolving interfaces, as it was done previously for the dewetting liquid films from solid substrates, see e.g. \cite{MWW06,KMW08}. 

As pointed out in the beginning of our study, mathematical theory for two-layer liquid flows leaves still many open questions and problems to be addressed. The present work can only be considered as a first step. Moreover, even considering only stationary solutions, we note that the general picture is much richer compared to the situation on a solid substrate, with energy structures leading to phase-inverted or more complicated patterns, and is be subject of our ongoing research. 

\section*{Acknowledgements}
The authors are grateful for the financial support of the DFG SPP 1506 ``Transport at fluid interfaces'' for financial support. GK thanks the Max-Planck-Institute for Mathematics in the Natural Sciences, Leipzig for the postdoctoral scholarship. DP acknowledges the financial support by DFG Research Center Matheon. The authors also enjoyed numerous fruitful discussion with Andreas M\"unch. 

\bibliographystyle{unsrtnat}
\bibliography{2layer-statsol}

\begin{thebibliography}{32}
\providecommand{\natexlab}[1]{#1}
\providecommand{\url}[1]{\texttt{#1}}
\expandafter\ifx\csname urlstyle\endcsname\relax
  \providecommand{\doi}[1]{doi: #1}\else
  \providecommand{\doi}{doi: \begingroup \urlstyle{rm}\Url}\fi

\bibitem[Sharma and Reiter(1996)]{SR96}
A.~Sharma and G.~Reiter.
\newblock Instability of thin polymer films on coated substrates: Rupture,
  dewetting and drop formation.
\newblock \emph{\em J. Colloid Interface Sci.}, 178:\penalty0 383--389, 1996.

\bibitem[Seemann et~al.(2001)Seemann, Herminghaus, and Jacobs]{SHJ01b}
R.~Seemann, S.~Herminghaus, and K.~Jacobs.
\newblock Gaining control of pattern formation of dewetting films.
\newblock \emph{Journal of Physics: Condensed Matter}, 13:\penalty0 4925--4938,
  2001.

\bibitem[Craster and Matar(2009)]{CO09}
R.~V. Craster and O.~K. Matar.
\newblock Dynamics and stability of thin liquid films.
\newblock \emph{\em Rev. Mod. Phys.}, 81:\penalty0 1131--1198, 2009.

\bibitem[Herminghaus et~al.(2008)Herminghaus, Brinkmann, and Seemann]{HBS08}
S.~Herminghaus, M.~Brinkmann, and R.~Seemann.
\newblock Wetting and dewetting of complex surface geometries.
\newblock \emph{Annual Review of Materials Research}, 38:\penalty0 101--121,
  2008.

\bibitem[Segalman and Green(1999)]{segalman1999dynamics}
R.A. Segalman and P.F. Green.
\newblock {Dynamics of rims and the onset of spinodal dewetting at
  liquid/liquid interfaces}.
\newblock \emph{Macromolecules}, 32\penalty0 (3):\penalty0 801--807, 1999.

\bibitem[Lambooy et~al.(1996)Lambooy, Phelan, Haugg, and
  Krausch]{lambooy1996dewetting}
P.~Lambooy, K.C. Phelan, O.~Haugg, and G.~Krausch.
\newblock Dewetting at the liquid-liquid interface.
\newblock \emph{Physical review letters}, 76\penalty0 (7):\penalty0 1110--1113,
  1996.

\bibitem[Slep et~al.(2000)Slep, Asselta, Rafailovich, Sokolov, Winesett, Smith,
  Ade, and Anders]{slep00}
D.~Slep, J.~Asselta, M.~H. Rafailovich, J.~Sokolov, D.~A. Winesett, A.~P.
  Smith, H.~Ade, and S.~Anders.
\newblock {Effect of an Interactive Surface on the Equilibrium Contact Angles
  in Bilayer Polymer Films}.
\newblock \emph{Langmuir}, 16:\penalty0 2369--2375, 2000.

\bibitem[Wang et~al.(2001)Wang, Krausch, and Geoghegan]{wang2001dewetting}
C.~Wang, G.~Krausch, and M.~Geoghegan.
\newblock {Dewetting at a Polymer- Polymer Interface: Film Thickness
  Dependence}.
\newblock \emph{Langmuir}, 17\penalty0 (20):\penalty0 6269--6274, 2001.

\bibitem[Neumann(1894)]{neumann1894vorlesung}
F.E. Neumann.
\newblock {Vorlesung {\"u}ber die Theorie der Capillarit{\"a}t}.
\newblock \emph{BG Teubner: Leipzig}, pages 113--116, 1894.

\bibitem[Brochard-Wyart et~al.(1993)Brochard-Wyart, Martin, and
  Redon]{brochard1993liquid}
F.~Brochard-Wyart, P.~Martin, and C.~Redon.
\newblock Liquid/liquid dewetting.
\newblock \emph{Langmuir}, 9\penalty0 (12):\penalty0 3682--3690, 1993.

\bibitem[Danov et~al.(1998)Danov, Paunov, Alleborn, Raszillier, and
  Durst]{danov1998stability}
K.D. Danov, V.N. Paunov, N.~Alleborn, H.~Raszillier, and F.~Durst.
\newblock Stability of evaporating two-layered liquid film in the presence of
  surfactant--i. the equations of lubrication approximation.
\newblock \emph{Chemical engineering science}, 53\penalty0 (15):\penalty0
  2809--2822, 1998.

\bibitem[Pototsky et~al.(2004)Pototsky, Bestehorn, Merkt, and
  Thiele]{Pototsky2004Alternative}
A.~Pototsky, M.~Bestehorn, D.~Merkt, and U.~Thiele.
\newblock Alternative pathways of dewetting for a thin liquid two-layer film.
\newblock \emph{Phys. Rev. E}, 70\penalty0 (2):\penalty0 025201, Aug 2004.

\bibitem[Fisher and Golovin(2005)]{fisher2005nonlinear}
L.S. Fisher and A.A. Golovin.
\newblock Nonlinear stability analysis of a two-layer thin liquid film:
  Dewetting and autophobic behavior.
\newblock \emph{Journal of colloid and interface science}, 291\penalty0
  (2):\penalty0 515--528, 2005.

\bibitem[Pototsky et~al.(2005)Pototsky, Bestehorn, Merkt, and
  Thiele]{pototsky2005morphology}
A.~Pototsky, M.~Bestehorn, D.~Merkt, and U.~Thiele.
\newblock Morphology changes in the evolution of liquid two-layer films.
\newblock \emph{The Journal of chemical physics}, 122:\penalty0 224711, 2005.

\bibitem[Craster and Matar(2006)]{craster2006dynamics}
R.V. Craster and O.K. Matar.
\newblock On the dynamics of liquid lenses.
\newblock \emph{Journal of colloid and interface science}, 303\penalty0
  (2):\penalty0 503--516, 2006.

\bibitem[Bandyopadhyay and Sharma(2006)]{bandyopadhyay2006nonlinear}
D.~Bandyopadhyay and A.~Sharma.
\newblock Nonlinear instabilities and pathways of rupture in thin liquid
  bilayers.
\newblock \emph{The Journal of chemical physics}, 125:\penalty0 054711, 2006.

\bibitem[Kriegsmann and Miksis(2003)]{kriegsmann2003steady}
J.J. Kriegsmann and M.J. Miksis.
\newblock Steady motion of a drop along a liquid interface.
\newblock \emph{SIAM Journal on Applied Mathematics}, pages 18--40, 2003.

\bibitem[Kostourou et~al.(2010)Kostourou, Peschka, M{\"u}nch, Wagner,
  Herminghaus, and Seemann]{kostourou2010interface}
K.~Kostourou, D.~Peschka, A.~M{\"u}nch, B.~Wagner, S.~Herminghaus, and
  R.~Seemann.
\newblock Interface morphologies in liquid/liquid dewetting.
\newblock \emph{Chemical Engineering and Processing: Process Intensification},
  2010.

\bibitem[Bertozzi et~al.(2001)Bertozzi, Gr{\"u}n, and
  Witelski]{bertozzi2001dewetting}
A.L. Bertozzi, G.~Gr{\"u}n, and T.P. Witelski.
\newblock Dewetting films: bifurcations and concentrations.
\newblock \emph{Nonlinearity}, 14:\penalty0 1569--1592, 2001.

\bibitem[Laugesen and Pugh(2000)]{laugesen2000linear}
RS~Laugesen and MC~Pugh.
\newblock {Linear stability of steady states for thin film and Cahn-Hilliard
  type equations}.
\newblock \emph{Archive for Rational Mechanics and Analysis}, 154\penalty0
  (1):\penalty0 3--51, 2000.

\bibitem[Otto et~al.(2007)Otto, Rump, and Slep{\v{c}}ev]{otto2007coarsening}
F.~Otto, T.~Rump, and D.~Slep{\v{c}}ev.
\newblock {Coarsening rates for a droplet model: Rigorous upper bounds}.
\newblock \emph{SIAM Journal on Mathematical Analysis}, 38\penalty0
  (2):\penalty0 503--529, 2007.

\bibitem[Kitavtsev et~al.(2010)Kitavtsev, Recke, and Wagner]{KRWP11}
G.~Kitavtsev, L.~Recke, and B.~Wagner.
\newblock Asymptotics for the spectrum of a thin film equation in a singular
  limit.
\newblock \emph{WIAS preprint}, 2010.

\bibitem[Zhang(2009)]{zhang2009counting}
Y.~Zhang.
\newblock {Counting the stationary states and the convergence to equilibrium
  for the 1-D thin film equation}.
\newblock \emph{Nonlinear Analysis}, 71\penalty0 (5-6):\penalty0 1425--1437,
  2009.

\bibitem[Oron et~al.(1997)Oron, Davis, and Bankoff]{ODB97}
A.~Oron, S.~H. Davis, and S.~G. Bankoff.
\newblock Long-scale evolution of thin liquid films.
\newblock \emph{Rev. Mod. Phys.}, 69\penalty0 (3):\penalty0 931--980, 1997.

\bibitem[Jachalski et~al.(2011)Jachalski, M\"unch, Peschka, and Wagner]{JMPW11}
S.~Jachalski, A.~M\"unch, D.~Peschka, and B.~Wagner.
\newblock Thin film models for two-layer flows with interfacial slip.
\newblock WIAS Preprint Number, 2011.

\bibitem[Lagerstrom(1988)]{paco88}
P.~A. Lagerstrom.
\newblock \emph{{Matched Asymptotic Expansions: Ideas and Techniques}}.
\newblock Springer Verlag, 1988.

\bibitem[Leoni(2009)]{leoni2009first}
G.~Leoni.
\newblock \emph{{A first course in Sobolev spaces}}.
\newblock American Mathematical Society, 2009.
\newblock ISBN 0821847686.

\bibitem[Lieb and Loss(2001)]{lieb2001analysis}
E.H. Lieb and M.~Loss.
\newblock \emph{Analysis}, volume~14.
\newblock American Mathematical Society, 2001.
\newblock ISBN 0821827839.

\bibitem[Braides(2002)]{braides2002gamma}
A.~Braides.
\newblock \emph{{Gamma-convergence for Beginners}}.
\newblock Oxford University Press, 2002.
\newblock ISBN 0198507844.

\bibitem[Dal~Maso(1993)]{dal1993introduction}
G.~Dal~Maso.
\newblock \emph{{Introduction to Gamma-convergence}}.
\newblock Birkh{\"a}user, 1993.
\newblock ISBN 081763679X.

\bibitem[M\"unch et~al.(2006)M\"unch, Wagner, and Witelski]{MWW06}
A.~M\"unch, B.~Wagner, and T.~P. Witelski.
\newblock Lubrication models with small to large slip lengths.
\newblock \emph{J. Engr. Math.}, 53:\penalty0 359--383, 2006.

\bibitem[King et~al.(2008)King, M\"unch, and Wagner]{KMW08}
J.~R. King, A.~M\"unch, and B.~Wagner.
\newblock Linear stability analysis of a sharp-interface model for dewetting
  thin films.
\newblock \emph{J. Engrg. Math.}, 63:\penalty0 177--195, 2008.

\end{thebibliography}
\clearpage
\addtocounter{tocdepth}{2}

\end{document}